\documentclass[10pt]{amsart}
\usepackage{latexsym,amssymb,enumerate,graphicx,psfrag,mathrsfs}

\begin{document}

\theoremstyle{plain}
  \newtheorem{theorem}{Theorem}[section]
  \newtheorem{proposition}[theorem]{Proposition}
  \newtheorem{lemma}[theorem]{Lemma}
  \newtheorem{corollary}[theorem]{Corollary}
  \newtheorem{conjecture}[theorem]{Conjecture}
\theoremstyle{definition}
  \newtheorem{definition}[theorem]{Definition}
  \newtheorem{example}[theorem]{Example}
  \newtheorem{observation}[theorem]{Observation}
  \newtheorem{observations}[theorem]{Observations}
  \newtheorem{question}[theorem]{Question}
 \theoremstyle{remark}
  \newtheorem{remark}[theorem]{Remark}

\numberwithin{equation}{section}
\def\ZZ{{\mathbb Z}}
\def\QQ{{\mathbb Q}}
\def\CC{{\mathbb C}}
\def\NN{{\mathbb N}}

\def\Kbar{{\overline{K}}}

\def\Y{{\mathbf Y}}
\def\SY{{S\mathbf Y}}
\def\YF{{\mathbf Y F}}
\def\il{{\text{Smith-eigenvalue}}}
\def\pd{{\frac{\partial}{\partial p_1}}}
\def\K{{\mathbb C}}
\def\E{{\mathcal E}}
\def\A{{\mathscr A}}
\def\P{{\mathbf P}}
\def\Q{{\mathbf Q}}
\def\v{v}
\def\GL{{\mathrm{GL}}}

\def\symm{\mathfrak{S}}

\def\coker{\mathrm{coker}}
\def\im{\mathrm{im}}
\def\diag{\mathrm{diag}}
\def\rank{\mathrm{rank}}
\def\Res{\mathrm{Res}}
\def\Ind{\mathrm{Ind}}
\def\triv{t}
\def\Irr{\mathrm{Irr}}
\def\sgn{s}
\def\ones{{\mathit{ones}}}
\def\Mat{\mathrm{Mat}}

\def\Class{{\mathbf{Cl}}}

\title[Differential posets have strict rank growth]{Differential posets have strict rank growth:\\
a conjecture of Stanley}
\author{Alexander R. Miller}
\email{mill1966@math.umn.edu}
\address{ School of Mathematics\\
University of Minnesota\\
Minneapolis, MN 55455}

\begin{abstract}  
We establish strict growth for the rank function of an 
$r$-differential poset.  We do so by exploiting the representation theoretic 
techniques developed by Reiner and the author~\cite{MR} for studying related Smith forms.
\end{abstract}

\thanks{Supported by NSF grant DMS-1001933.}
\maketitle

\section{Introduction}
For a positive integer $r$, an \emph{$r$-differential poset} is a graded poset $P$ with a minimum 
element, having all intervals and all rank cardinalities finite, satisfying 
\begin{enumerate}[(D1)]
\item  If an element of $P$ covers $m$ others, then it will be covered by $m+r$ others.
\item  If two elements of $P$ have exactly $m$ elements that they both cover, then there will be exactly $m$ 
elements that cover them both.
\end{enumerate}
We write $P_n$ for the $n$th rank of $P$ and set $p_n:=|P_n|$.  
  Considering the free $\mathbb Z$-module $\mathbb ZP_n\cong \mathbb Z^{p_n}$ generated by the elements of $P_n$, 
define the \emph{up} and \emph{down} maps
\begin{align*}
&U_n(=D_{n+1}^t):\mathbb ZP_n\to\mathbb ZP_{n+1}\\
&D_n(=U_{n-1}^t):\mathbb ZP_n\to\mathbb ZP_{n-1}
\end{align*}
in which a basis element is sent by $U_n$ (resp. $D_n$) to the sum of all elements that cover (resp. are covered by) it.  We shall 
often omit subscripts when the domain is clear.  Setting
\[UD_n:=U_{n-1}D_n\quad\text{and}\quad DU_n:=D_{n+1}U_n,\]
conditions (D1,D2) can be rephrased as saying that
\[DU_n-UD_n=rI\quad\text{for each $n\geq 1$}.\]

The following result of Stanley concerning the spectra of $DU_n$ is central to the theory 
and plays an important role in our analysis.

\begin{theorem}[Stanley~\cite{S}]\label{eig}  Let $P$ be an $r$-differential poset.  Then 
\[\det(DU_n+tI)=\prod_{i=0}^n (t+r(i+1))^{\Delta p_{n-i}},\]
where $\Delta p_n:=p_n-p_{n-1}$.
\end{theorem}

Consequently, the rank sizes of a differential poset $P$ weakly increase, that is $p_0\leq p_1\leq p_2\leq \cdots$.  Though 
the growth of the rank function was recently studied in~\cite{SZ}, an answer to the basic initial question of Stanley asking whether the 
rank sizes {\it strictly} increase has remained elusive; see~\cite{MR,S,SZ}.  The purpose of this paper is to resolve this question by establishing

\begin{conjecture}[Stanley~\cite{S}]  An $r$-differential poset $P$ has $p_1<p_2<p_3<\cdots$.
\end{conjecture}

The techniques of this paper were inspired by those developed by Reiner and the author in~\cite{MR} for towers of group algebras, 
to which we refer the reader for additional background, notation, and motivation.  The key observation being that one can mimic enough 
of the character theory used in~\cite{MR} with just two distinct elements $t,s\in P_n$ having $D^nt=D^ns=\hat{0}$, 
the bottom element, for $n\geq 2$.

\section{The proof}

We start with the following fundamental observation.

\begin{proposition}[Miller-Reiner~\cite{MR}]
\label{prop:join-irreducible-chains}
Let $P$ be an $r$-differential poset, and $j_{m-1} \lessdot j_m$ any
covering pair in $P$ with the property that 
$j_{m-1}, j_m$ both cover at most one element of $P$.  

Then for any integer $n \geq m$ one can extend this to a saturated chain 
$$
j_{m-1} \lessdot j_m \lessdot j_{m+1} \lessdot \cdots \lessdot j_{n-1} \lessdot j_n
$$
in which each $j_\ell$ covers at most one element of $P$.
\end{proposition}

Noting that $r$ elements cover the bottom element of an $r$-differential poset and that all $1$-differential posets 
are isomorphic to Young's lattice $\Y$ up to rank $n=2$ (see Figure~\ref{Fig:Chains}), we have the following

\begin{corollary}\label{Cor:Def}
Let $P$ be an $r$-differential poset.  Then there exists a pair of chains
$$
\triv_{0} \lessdot \triv_1 \lessdot\triv_{2} \lessdot \cdots \lessdot \triv_n \lessdot \cdots
\qquad\text{and}\qquad
\sgn_{0} \lessdot \sgn_1 \lessdot \sgn_{2} \lessdot \cdots \lessdot \sgn_n \lessdot \cdots
$$
with the following properties:
\begin{enumerate}[(i)]
\item  $\rank(\triv_n)=\rank(\sgn_n)=n$;
\item  if $r=1$ then $\triv_n\neq\sgn_n$ for $n\geq 2$, while $\triv_0=\sgn_0$ and $\triv_1=\sgn_1$;
\item  if $r>1$ then $\triv_n\neq\sgn_n$ for $n\geq 1$, while $\triv_0=\sgn_0$;
\item  each $\triv_n$ and $\sgn_n$ covers at most one element of $P$.
\end{enumerate}
\end{corollary}

Fix one such pair of chains for each differential poset $P$, and refer to each using the notation of 
Corollary~\ref{Cor:Def}; see Figure~\ref{Fig:Chains}.    
Further, when considering the matrix of an operator $A\in{\mathrm{End}}_{\mathbb Z}(\mathbb Z P_n)$, 
it is understood that the \emph{standard basis} consisting 
of the elements of $P_n$ is to be considered, and ordered so 
that $\triv_n$ indexes the first row and column of the matrix.  
Lastly, we remark that the notation is motivated by Young's lattice, 
viewed as the Bratteli diagram associated to the tower $\{\mathbb C\mathfrak S_n\}_{n\geq 0}$, 
where $\triv_n$ and $\sgn_n$ correspond to the two linear representations of $\mathfrak S_n$ for $n\geq 2$.

\begin{figure}[hbt]
\includegraphics[width=\textwidth]{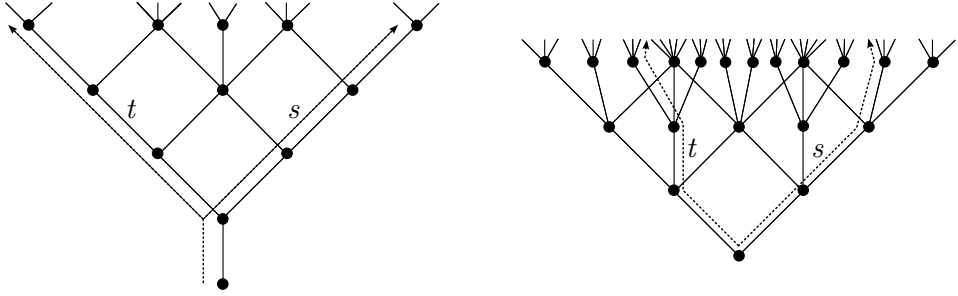}
\caption{An illustration of Corollary~\ref{Cor:Def} for a $1$-differential poset (left) and a $2$-differential poset (right).}\label{Fig:Chains}
\end{figure}

For nonnegative integers $r,k,\ell$, define 
\[\ell!_{r,k}:=(r\cdot\ell+k)(r\cdot(\ell-1)+k)\cdots(r\cdot 1+k)\]
with $0!_{r,k}:=1$.
\begin{proposition}\label{Prop:Fundamental}
Let $P$ be an $r$-differential poset, $k$ be a positive integer, and 
\begin{equation}\v_{n,k}:=\sum_{j=0}^n (-1)^j \frac{U^j \triv_{n-j}}{(j+1)!_{r,k}}.\label{Eqn:Fundamental}\end{equation}
Then $(DU_n+kI) \v_{n,k}=\triv_n$.
\end{proposition}

\begin{proof}
For $n=0$, we indeed have $(DU_0+kI)\frac{\triv_0}{r+k}=\triv_0$.  Inducting on $n$,
\begin{align*}
(DU_{n+1}+kI) \v_{n+1,k} 
&=(UD+(r+k)I)\left(\frac{\triv_{n+1}}{r+k}+\sum_{j=1}^{n+1}(-1)^j\frac{U^j \triv_{n+1-j}}{(j+1)!_{r,k}}\right)\\
&=
   (UD+(r+k)I)\left( \frac{\triv_{n+1}}{r+k}-\frac{U}{r+k} \sum_{\ell=0}^{n}(-1)^\ell \frac{U^\ell \triv_{n-\ell}}{(\ell+1)!_{r,r+k}} \right)\\
&= \frac{1}{r+k}\left((UD+(r+k)I) \triv_{n+1}-
U(DU+(r+k)I) \v_{n,r+k}\right)\\
&= \frac{1}{r+k}\left((UD+(r+k)I) \triv_{n+1}-
U\triv_{n}\right)\\
&=\frac{1}{r+k}\left(U\triv_n+(r+k)\triv_{n+1}-U\triv_n\right)\\
&=\triv_{n+1},
\end{align*}
where the fourth equality is by induction, and the fifth follows from Corollary~\ref{Cor:Def}.
\end{proof}

\begin{corollary}\label{Cor:Column}
Let $P$ be an $r$-differential poset, and let $k$ be a positive integer.  Then $DU_n+kI$ is invertible and 
the column vector of $\v_{n,k}$ forms the first column of the inverse $(DU_n+kI)^{-1}$.
\end{corollary}

\begin{proof}
The first claim follows from Theorem~\ref{eig}, and the second follows from Proposition~\ref{Prop:Fundamental} 
and our convention of ordering the basis so that $\triv_n$ indexes the first row and column of $DU_n+kI$.
\end{proof}

An integral matrix 
$D=(d_{ij})\in\Mat_{n\times m}(\mathbb Z)$ 
is said to be \emph{in Smith form} if it is diagonal in the sense that 
$d_{ij}=0$ for $i\neq j$, 
and its diagonal entries are nonnegative and satisfy 
$d_{11}|d_{22}|\cdots|d_{\min\{n,m\}}$, 
in which case we set 
$d_i:=d_{ii}$ for each $i$.  
Recall that every integral matrix 
$A\in\Mat_{n\times m}(\mathbb Z)$ 
can be brought into Smith form by an appropriate change of basis in 
$\mathbb Z^n$ and $\mathbb Z^m$, 
i.e. there exist matrices 
$P\in \GL_n(\mathbb Z)$ and $Q\in\GL_m(\mathbb Z)$ 
for which $PAQ=D$ is in Smith form.  
And though the matrices $P,Q$ are not 
necessarily unique, the resulting Smith form $D$ 
is, and its entries $d_1,\ldots, d_{\min\{n,m\}}$ are 
called the \emph{Smith entries of $A$}, with 
$d_{\min\{n,m\}}$ referred to as the \emph{last Smith entry}.
When $A$ is square and invertible, we have the following 
well-known characterization of this last entry.

\begin{proposition}[cf. \cite{MR}]\label{Prop:Inv}
Let $A$ be an $n\times n$ invertible (over $\mathbb Q$) integral matrix, 
and let $s$ be the smallest positive integer for which $s\cdot A^{-1}$ is integral.  
Then $d_n=s$.
\end{proposition}

\begin{theorem}
\label{theorem:last-smith-entry}
Let $P$ be an $r$-differential poset, let $k$ and $n$ be positive integers, and let $d_{p_n}$ denote the 
last Smith entry of $DU_n+kI$.  Then the following hold:
\begin{enumerate}[(i)]
\item $(n+1)!_{r,k}$ divides $d_{p_n}$ if $r\geq 2$;
\item $(n-1)!_{r,k}\cdot (n+1+k)$ divides $d_{p_n}$ if $r =1$.
\end{enumerate}
\end{theorem}

\begin{proof}  
By Proposition~\ref{Prop:Inv}, $d_{p_n}$ is the smallest positive integer for which 
\[d_{p_n}\cdot(DU_n+kI)^{-1}\in \Mat_{p_n\times p_n}(\mathbb Z).\]
It thus suffices to show that the claimed divisor $d$ of $d_{p_n}$ is the 
smallest positive integer $s$ for which the first column of $s\cdot (DU_n+kI)^{-1}$ is integral, 
or equivalently $s\cdot \v_{n,k}\in \mathbb Z P_n$ by Corollary~\ref{Cor:Column}.

For $r\geq 2$ it is clear that $d\cdot\v_{n,k}\in\mathbb Z P_n$ by~\eqref{Eqn:Fundamental}.  Moreover, by Corollary~\ref{Cor:Def}
\[
\left\langle\, U^j \triv_{n-j}\ ,\ \sgn_n\, \right\rangle 
=\left\langle\, \triv_{n-j}\ ,\ D^j\sgn_n\, \right\rangle
=\left\langle\, \triv_{n-j}\ ,\ \sgn_{n-j}\, \right\rangle=
\begin{cases}
1 & \text{ if $j=n$}\\
0 & \text{ if $j\leq n-1$,}
\end{cases}
\]
where $\langle -,-\rangle$ denotes the bilinear form obtained by decreeing that $\langle x,y\rangle=\delta_{x,y}$ for $x,y\in P_n$.  It follows that $d$ is the smallest positive integer 
for which $d\cdot\v_{n,k}\in\mathbb Z P_n$.

Similarly, for $r=1$ we have that
\begin{equation}\label{Eq:r1}
\left\langle\, U^j \triv_{n-j}\ ,\ \sgn_n\, \right\rangle =
\begin{cases}
1 & \text{ if $j=n$ or $n-1$}\\
0 & \text{ if $j\leq n-2$.}
\end{cases}
\end{equation}
Considering the $j=n-1$ and $j=n$ terms of $\v_{n,k}$, we have
\[\frac{U^{n-1} \triv_{1}}{n!_{1,k}}-\frac{U^n \triv_{0}}{(n+1)!_{1,k}}
= \frac{U^{n} \triv_{0}}{n!_{1,k}}- \frac{U^n \triv_{0}}{(n+1)!_{1,k}}
= \frac{U^n\triv_0}{(n-1)!_{1,k}\cdot (n+1+k)},
\]
and thus 
\begin{equation}\label{Eq:r1:2}
\v_{n,k}=\sum_{j=0}^{n-2} (-1)^j \frac{U^j \triv_{n-j}}{(j+1)!_{1,k}}+(-1)^{n-1} \frac{U^n\triv_0}{(n-1)!_{1,k}\cdot (n+1+k)}.
\end{equation}
It follows that indeed $d\cdot\v_{n,k}\in\mathbb Z P_n$.  Further,~\eqref{Eq:r1} and~\eqref{Eq:r1:2} together imply that 
$d$ is the smallest such integer.
\end{proof}

\begin{corollary}
Let $P$ be an $r$-differential poset.  Then $p_1<p_2<p_3<\cdots$.
\end{corollary}

\begin{proof}  The following argument was given by Reiner and the author in~\cite{MR}, verbatim, but in the context of 
differential posets associated with towers of group algebras; see~\cite[Proof of Cor. 6.13]{MR}.  

For a fixed $n \geq 2$, to show that $\Delta p_n > 0$, 
apply Theorem~\ref{theorem:last-smith-entry} with the positive integer $k$
chosen so that $r+k$ is a prime that divides {\it none}
of 
$$
2r+k, \, 3r+k, \, \ldots, \, (n+1)r+k
$$
(e.g. pick $p$ to be a
prime larger than $(n+1)r$ and take $k=p-r$).
Theorem~\ref{theorem:last-smith-entry} tells us that the prime $r+k$ divides the last 
Smith form entry $d_{p_n}$ for $DU_n+kI$ over $\ZZ$,
and hence it also divides 
$$
\det(DU_n+kI) = 
(r+k)^{\Delta p_n} (2r+k)^{\Delta p_{n-1}} \cdots (nr+k)^{\Delta p_1} ((n+1)r+k)^{p_0}
$$
according to Theorem~\ref{eig}.
Since $r+k$ is a prime that can only divide the first factor on the right, it must be
that $\Delta p_n > 0$.
\end{proof}

\section{Acknowledgements}
The author would like to thank Kevin Dilks for carefully reviewing this work, and his advisor 
Victor Reiner for helpful edits.


\end{document}